\RequirePackage{fix-cm}
\documentclass[smallextended]{svjour3}       
\smartqed  

\RequirePackage[OT1]{fontenc}
\usepackage{mathrsfs,url,color}
\RequirePackage[colorlinks,citecolor=blue,urlcolor=blue,pagebackref=true]{hyperref}

\usepackage{graphicx}
\usepackage{amsmath,amsfonts,epsfig,wasysym,enumerate,amssymb}

\makeatletter \@addtoreset{equation}{section}
\renewcommand{\thesection}{\arabic{section}}

\usepackage{latexsym}

\newtheorem{thm}{Theorem}[section]
\newtheorem{prop}[thm]{Proposition}
\newtheorem{lem}[thm]{Lemma}

\newtheorem{rem}[thm]{Remark}
\newtheorem{cor}[thm]{Corollary}

\newtheorem{exm}[thm]{Example}
\newcommand{\thmref}[1]{Theorem~{\rm \ref{#1}}}
\newcommand{\lemref}[1]{Lemma~{\rm \ref{#1}}}
\newcommand{\cororef}[1]{Corollary~{\rm \ref{#1}}}
\newcommand{\propref}[1]{Proposition~{\rm \ref{#1}}}

\begin{document}

\title{The Marcinkiewicz--Zygmund-Type Strong Law of Large Numbers with General Normalizing Sequences 
\thanks{The paper was supported by NAFOSTED, Grant No. 101.03-2015.11.}
}

\titlerunning{Marcinkiewicz--Zygmund Strong Law}

\author{Vu T. N. Anh \and Nguyen T. T. Hien \and Le V. Thanh$^*$\thanks{$^*$Corresponding author}          \and Vo T. H. Van
}


\institute{V. T. N. Anh \at
              Department
of Mathematics, Hoa Lu University, Ninh Binh, Vietnam\\
              \email{vtnanh@hluv.edu.vn}         
           \and
           N. T. T. Hien \at
              Department of Mathematics, Vinh University, Nghe An, Vietnam\\
              \email{hienntt.ktoan@vinhuni.edu.vn}
           \and 
           L. V. Thanh \at 
              Department of Mathematics, Vinh University, Nghe An, Vietnam\\
              \email{levt@vinhuni.edu.vn}
           \and
           V. T. H. Van \at 
              Department of Mathematics, Vinh University, Nghe An, Vietnam\\
              \email{vanvth@vinhuni.edu.vn}
           }

\date{}

\maketitle

\begin{abstract}
	\quad

\noindent This paper establishes complete convergence
for weighted sums and the Marcinkiewicz--Zygmund-type strong law of large
numbers for sequences of negatively associated and identically distributed random variables
$\{X,X_n,n\ge1\}$ with general 
normalizing constants under a moment condition that $ER(X)<\infty$,
where $R(\cdot)$ is a regularly varying
function.
The result is new even when the random
variables are independent and identically distributed (i.i.d.), and a special case of this result comes
close to a solution to an open question raised by Chen and Sung (Statist Probab Lett 92:45--52, 2014). 
The proof exploits some properties of slowly
varying functions and the de Bruijin conjugates.
A counterpart of the main result obtained by 
Martikainen (J Math Sci 75(5):1944--1946, 1995) on the
Marcinkiewicz--Zygmund-type strong law of large numbers for pairwise i.i.d. random variables
is also presented. Two illustrated examples are provided, including a strong
law of large numbers for pairwise negatively dependent random variables
which have the same distribution as the random variable appearing in the St. Petersburg game.


\end{abstract}

\section{Introduction}\label{sec:int}

The motivation of this paper is an open
question raised recently by Chen and Sung \cite{ChenSung14}.
Let $1<\alpha\le 2$, $\gamma>0$ and let $\{X,X_n,n\ge 1\}$ be a sequence of
negatively associated and identically distributed random variables with $E(X)=0$.
Sung \cite{Sung11} 
proved that if
\begin{equation}
\begin{cases}
E|X|^\gamma<\infty \text{ for }\ \gamma>\alpha,\\
E|X|^\alpha\log(|X|+2) <\infty \text{ for }\ \gamma=\alpha,\\
E|X|^\alpha<\infty \text{ for }\ \gamma<\alpha,
\end{cases}
\end{equation}
then
\begin{equation}\label{ChenSung03}
\sum_{n=1}^\infty 
 n^{-1} P\left(\max_{1\le k\le n}\left|
\sum_{i=1}^{k}a_{ni}X_i\right|>\varepsilon  n^{1/\alpha}\log^{1/\gamma}(n)\right)<\infty
\ \text{ for all } \ \varepsilon >0,
\end{equation}
where
$\{a_{ni},n\ge 1,1\le i\le n\}$
are constants satisfying 
\begin{equation}\label{ChenSung05}
\sup_{n\ge 1}\dfrac{\sum_{i=1}^{n}|a_{ni}|^{\alpha}}{n}<\infty.
\end{equation}
Here and thereafter, $\log$ denotes the logarithm to the base $2$.
Chen and Sung \cite{ChenSung14} proved that for the case where $\gamma>\alpha$, the condition $E|X|^\gamma<\infty$
is optimal. They raised an open question about finding the optimal condition for \eqref{ChenSung03} 
when $\gamma\le \alpha$.
For the case where $\gamma< \alpha$, Chen and Sung \cite[Corollary 2.2]{ChenSung14} proved
that \eqref{ChenSung03} holds under an almost optimal condition that
\[E\left(|X|^\alpha\log^{1-\alpha/\gamma}(|X|+2)\right) <\infty.\]

In this note, by
using some results related to regularly varying functions,
we provide the necessary and sufficient conditions for
\begin{equation}\label{ChenSung07}
\sum_{n}
 n^{-1} P\left(\max_{1\le k\le n}\left|
\sum_{i=1}^{k}a_{ni}X_i\right|>\varepsilon n^{1/\alpha}\tilde{L}(n^{1/\alpha})\right)<\infty
\ \text{ for all } \ \varepsilon >0,
\end{equation}
where $\tilde{L}(\cdot)$ is 
the de Bruijn conjugate of a slowly varying function $L(\cdot)$, defined on $[A,\infty)$
for some $A>0$. 
This result is new even when the random variables are i.i.d. 
By letting $L(x)\equiv \log^{-1/\gamma}(x),\ x\ge 2$, we obtain
optimal moment condition for \eqref{ChenSung03}. 

Weak laws of large numbers with the norming constants are of the form 
$n^{1/\alpha}\tilde{L}(n^{1/\alpha})$ were studied by Gut \cite{Gut04}, and 
Matsumoto and Nakata \cite{MN13}. 
The Marcinkiewicz--Zygmund strong law of large numbers has been extended and generalized in many directions
by a number of authors, see \cite{BaKa,DedeckerMerlevede,GutStadmueller,HechnerHeinkel,Rio95a,Rio95b,Szewczak}
and references therein. To our best knowledge, there is not any result in the literature 
that considers strong law of large numbers with general
normalizing constants $n^{1/\alpha}\tilde{L}(n^{1/\alpha})$ 
except Gut and Stadm\"{u}ller \cite{GutStadmueller}
who studied the Kolmogorov strong law of large numbers, but for delay sums.
The main result of this paper fills this gap.
Recently, Miao et al. \cite{MMX} have studied the Marcinkiewicz--Zygmund-type
strong law of large numbers where the norming constants are of the form  $n^{1/\alpha}\log^{\beta/\alpha} n$
for some $\beta\ge0$, which is a special case of our result.

The concept of negative
association of random variables was introduced by Joag-Dev and Proschan \cite{JoPr}. 
A collection $\{X_1,\dots,X_n\}$ of
random variables is said to be negatively associated if for any
disjoint subsets $A,B$ of $\{1,\dots,n\}$ and any real
coordinatewise nondecreasing functions $f$ on ${\mathbb{R}}^{|A|}$
and $g$ on ${\mathbb{R}}^{|B|}$, 
\begin{equation}\label{neg-var} \text{Cov}(f(X_k,k\in
A),g(X_k,k\in B))\le 0
\end{equation} 
whenever the covariance exists,  where
$|A|$ denotes the cardinality of $A$.
A sequence $\{X_n,n\ge 1\}$ of random variables is said to be
negatively associated if every finite subfamily is negatively
associated.

There is a weaker concept of dependence called 
negative dependence, which was introduced by Lehmann \cite{Lehmann66} 
and further investigated
by Ebrahimi and Ghosh \cite{EG81} and Block et al. \cite{BSS}. 
A collection of
random variables $\{X_1,\dots,X_n\}$ is said to be negatively dependent if for all $x_1,\dots,x_n \in \mathbb{R}$,
$$P(X_1\le x_1, \dots, X_n\le x_n)\le P(X_1\le x_1) \dots P(X_n\le x_n),$$
and
$$P(X_1> x_1, \dots, X_n> x_n)\le P(X_1> x_1) \dots P(X_n> x_n).$$
A sequence of random variables $\{X_i,i\ge 1\}$ is said to be negatively dependent if
for any $n\ge 1$, the collection $\{X_1,\dots,X_n\}$ is negatively dependent.
 A sequence of random variables $\{X_i,i\ge 1\}$ is said to be pairwise negatively dependent if
for all $x,y\in \mathbb{R}$ and for all $i\not=j$,
\[P(X_i\le x,X_j\le y)\le P(X_i\le x) P(X_j\le y).\]
It is well known and easy to prove that 
$\{X_i,i\ge 1\}$ is pairwise negatively dependent if and only if
for all $x,y\in \mathbb{R}$ and for all $i\not=j$,
\begin{equation*}
P(X_i> x,X_j> y)\le P(X_i> x) P(X_j> y).
\end{equation*}

By Joag-Dev and Proschan \cite[Property P3]{JoPr}, negative association implies negative dependence.
For examples about negatively dependent random variables which are not negatively associated, see \cite[p. 289]{JoPr}.
Of course, pairwise independence implies pairwise negative dependence, but pairwise
independence and negative dependence do not imply each other.
Joag-Dev and Proschan \cite{JoPr} pointed out that many useful distributions enjoy the negative association properties 
(and therefore, they are negatively dependent)
including multinomial distribution, multivariate hypergeometric distribution, Dirichlet distribution, 
strongly Rayleigh distribution and
distribution of random sampling without replacement. 
Limit theorems for negatively associated and negatively dependent random variables have received extensive attention.
We refer to \cite{JiLi,Matula92,Shao00} and references therein.
These concepts of dependence can be extended to the Hilbert space-valued random variables; see, e.g.,
\cite{BDD,HTV,KKH,Thanh13}, among others.

The rest of the paper is arranged as follows. Section \ref{Pre}
presents some results on slowly varying functions needed in proving the main results. 
Section \ref{NA} focuses on complete convergence for
weighted sums of negatively associated and identically distributed random variables. 
In Sect. \ref{Pair}, we apply a result concerning slowly varying
functions developed in Sect. \ref{Pre} to give a counterpart of
Martikainen's strong law of large numbers (see \cite{Martikainen}) 
for sequences of pairwise negatively dependent and identically distributed random variables.
As an application, we prove a strong law
of large numbers for pairwise negatively dependent random variables
which have the same distribution as the random variable appearing in the St. Petersburg game.

\section{Some Facts Concerning Slowly Varying Functions}\label{Pre}
Some technical results concerning slowly varying functions will be presented in
this section.

The notion of regularly varying function can be found in \cite[Chapter 1]{Seneta76}.
A real-valued function $R(\cdot )$ is said to be regularly varying with index of regular variation
$\rho$ ($\rho\in\mathbb{R}$) if it is 
a positive and measurable function on $[A,\infty)$ for some $A> 0$, and for each $\lambda>0$,
\begin{equation}\label{rv01}
\lim_{x\to\infty}\dfrac{R(\lambda x)}{R(x)}=\lambda^\rho.
\end{equation}
A regularly varying function with the index of regular variation $\rho=0$ is called slowly varying.
It is well known that a function $R(\cdot )$ is regularly varying
with the index of regular variation $\rho$ if and only if it can be written in the form
\begin{equation}\label{sv01}
R(x)=x^\rho L(x)
\end{equation}
where $L(\cdot)$ is a slowly varying function (see, e.g., \cite[p. 2]{Seneta76}).
On the regularly 
varying functions and their important role in probability, 
we refer to Seneta \cite{Seneta76}, 
Bingham, Goldie and Teugels \cite{BGT}, and more recent
survey paper by Jessen and Mikosch \cite{JeMi}.
Regular variation is also one of the key notions for
modeling the behavior of large telecommunications networks; see, e.g., 
Heath et al. \cite{HRS}, Mikosch et al. \cite{MRRS}. 

The basic result in the theory of slowly varying functions is the representation theorem
(see, e.g., \cite[Theorem 1.3.1]{BGT})
which states that for a positive and measurable function $L(\cdot)$ defined on $[A,\infty)$ for some $A>0$,
$L(\cdot)$ is slowly varying if and only if it can be written in the form 
\[L(x)=c(x)\exp\left(\int_{B}^{x}\dfrac{\varepsilon(u)du}{u}\right)\]
for some $B\ge A$ and for all $x\ge B$, where $c(\cdot)$ is a positive bounded measurable function
defined on $[B,\infty)$ satisfying
$\lim_{x\to\infty}c(x)=c\in (0,\infty)$ and $\varepsilon(\cdot)$ is 
a continuous function defined on $[B,\infty)$ satisfying
$\lim_{x\to\infty}\varepsilon(x)=0$.
Seneta \cite{Seneta73} (see also in \cite[Lemma 1.3.2]{BGT}) proved that if 
$L(\cdot)$ is a slowly varying function defined on $[A,\infty)$ for some $A> 0$,
then there exists $B\ge A$ such that $L(x)$ is bounded on every finite closed interval $[a,b]\subset [B,\infty)$.

Let $L(\cdot)$ be a slowly varying function. Then by \cite[Theorem 1.5.13]{BGT},
there exists a slowly varying function $\tilde{L}(\cdot)$, unique up to asymptotic equivalence, satisfying
\begin{equation}\label{BGT1513}
\lim_{x\to\infty}L(x)\tilde{L}\left(xL(x)\right)=1\ \text{ and } \lim_{x\to\infty}\tilde{L}(x)L\left(x\tilde{L}(x)\right)=1.
\end{equation}
The function $\tilde{L}$ is called the de Bruijn conjugate of $L$, and $\left(L,\tilde{L}\right)$ is called a (slowly varying) conjugate pair (see, e.g., \cite[p. 29]{BGT}).
By \cite[Proposition 1.5.14]{BGT}, if $\left(L,\tilde{L}\right)$ is a conjugate pair, then for $a,b,\alpha>0$, 
each of $\left(L(ax),\tilde{L}(bx)\right)$, $\left(aL(x),a^{-1}\tilde{L}(x)\right),$
$\left(\left(L(x^\alpha)\right)^{1/\alpha},(\tilde{L}(x^\alpha))^{1/\alpha}\right)$
is a conjugate pair.
Bojani\'{c}, R. and Seneta \cite{Bojanic71} (see also Theorem 2.3.3 and Corollary 2.3.4 in \cite{BGT})
proved that if $L(\cdot)$ is a slowly varying function satisfying 
\begin{equation}
\lim_{x\to\infty}\left(\dfrac{L(\lambda_0 x)}{L(x)}-1\right)\log( L(x))=0,
\end{equation}
for some $\lambda_0>1$, then for every $\alpha\in\mathbb{R}$,
\begin{equation}\label{BGT2.3.4}
\lim_{x\to\infty}\dfrac{L(xL^\alpha(x))}{L(x)}=1,
\end{equation}
and therefore, we can choose (up to aymptotic equivalence) $\tilde{L}(x)=1/L(x)$.
In particular, if $L(x)=\log (x)$ then $\tilde{L}(x)=1/\log(x)$.

The following lemma follows from Theorem 1.5.12 and Proposition 1.5.15 in \cite{BGT}.
Here and thereafter, for a slowly varying function $L(\cdot)$ defined on $[A,\infty)$ for some $A>0$,
we denote the Brujin
conjugate of $L(\cdot)$
by $\tilde{L}(\cdot)$. Without loss of generality, we assume that
$\tilde{L}(\cdot)$ is also defined  on $[A,\infty)$, and that $L(x)$ and $\tilde{L}(x)$ are both bounded 
on finite closed intervals.

\begin{lem}\label{lemma00sv}
Let $\alpha,\beta>0$ and $L(\cdot)$ be a slowly varying function. Let $f(x)=x^{\alpha \beta}L^{\alpha}(x^\beta)$ and
$g(x)=x^{1/{(\alpha \beta)}}\tilde{L}^{1/\beta}(x^{1/\alpha})$. Then
\begin{equation}
\lim_{x\to\infty}\dfrac{f(g(x))}{x}=\lim_{x\to\infty}\dfrac{g(f(x))}{x}=1.
\end{equation}
\end{lem}

The second lemma shows that we can approximate a slowly varying function  $L(\cdot)$
by a differentiable slowly varying function $L_1(\cdot)$. See Galambos and Seneta \cite[p. 111]{GS73} for a proof.

\begin{lem}\label{lem:GS}
	For any
	slowly varying function  $L(\cdot)$ defined on $[A,\infty)$ for some $A> 0$, there exists
	a differentiable slowly varying function $L_1(\cdot)$ defined on $[B,\infty)$ for some $B\ge A$
	such that 
	\begin{equation*}
	\lim_{x\to\infty}\dfrac{L(x)}{L_{1}(x)}=1\ \text{ and }\ \lim_{x\to\infty}\dfrac{xL_{1}'(x)}{L_{1}(x)}=0.
	\end{equation*}
	
	Conversely, if  $L(\cdot)$ is a positive differentiable function satisfying  
	\begin{equation}\label{sv002}
	\lim_{x\to\infty}\dfrac{xL'(x)}{L(x)}=0,
	\end{equation}
	then $L(\cdot)$ is a slowly varying function. 
\end{lem}

Because of \lemref{lem:GS}, we can work with differentiable slowly varying functions $L(\cdot)$ that satisfy \eqref{sv002}
in our setting.

The proof of \lemref{lemma01sv} (i)  follows from direct calculations (by taking the derivative).
\lemref{lemma01sv} (ii) 
is an easy consequence of the representation theorem stated above.

\begin{lem}\label{lemma01sv} Let $p>0$ and let $L(\cdot)$ be a slowly varying function defined on $[A,\infty)$ for some $A> 0$,
satisfying \eqref{sv002}.
Then the following statements hold.
\begin{description}
\item{(i)} There exists $B\ge A$ such that $x^pL(x)$ is increasing on $[B,\infty)$, $x^{-p}L(x)$ is decreasing on $[B,\infty)$, and $\lim_{x\to\infty}x^pL(x)=\infty,\ \lim_{x\to\infty}x^{-p}L(x)= 0$.
\item{(ii)} For all $\lambda>0$,
\[\lim_{x\to\infty}\dfrac{L(x)}{L(x+\lambda)}=1.\]
\end{description}
\end{lem}

\begin{rem}\label{R11}
	{\rm
If we do not have the assumption that $L(\cdot)$ satisfies \eqref{sv002}, then we still have 
$x^pL(x)\rightarrow\infty,\ x^{-p}L(x)\rightarrow 0$ as $x\to\infty$ (see Seneta \cite[p. 18]{Seneta76}), but we do not
	have the monotonicity as in Lemma \ref{lemma01sv} (i).
 }
\end{rem}

The following lemma is a direct consequence of Karamata's theorem (see \cite[Theorem 1.5.10]{BGT}) as 
was so kindly pointed out to us by the referee.

\begin{lem}\label{sv51}
Let $p>1$, $q\in\mathbb{R}$ and  $L(\cdot)$ be a differentiable slowly varying function defined on $[A,\infty)$ for some $A>0$.
Then 
\begin{equation}\label{rem08}
\begin{split}
\sum_{k=n}^\infty \dfrac{L^q(k)}{k^p}\sim  \dfrac{L^q(n)}{(p-1)n^{p-1}}.
\end{split}
\end{equation} 
\end{lem}

The following proposition gives a criterion for $E\left(|X|^\alpha L^\alpha(|X|+A)\right)<\infty$.

\begin{prop}\label{lemma_bound01}
Let $\alpha\ge 1$, and let $X$ be a random variable.
Let $L(\cdot)$ be a slowly varying function defined on $[A,\infty)$ for some $A>0$. Assume that 
$x^{\alpha}L^\alpha(x)$ and $x^{1/\alpha}\tilde{L}(x^{1/\alpha})$
are increasing on $[A,\infty)$.
Then
\begin{equation}\label{lem_com020}
E\left(|X|^\alpha L^\alpha(|X|+A)\right)<\infty\ \text{ if and only if }\ \sum_{n\ge A^\alpha}P(|X|>b_n)<\infty,
\end{equation}
where $b_n=n^{1/\alpha}\tilde{L}\left(n^{1/\alpha}\right)$, $n\ge A^\alpha$.
\end{prop}

\begin{proof}
Let $f(x)=x^\alpha L^\alpha(x)$, $g(x)=x^{1/\alpha}\tilde{L}(x^{1/\alpha})$.
Since $L(\cdot)$ is positive and bounded on finite closed intervals, 
\[E\left(|X|^\alpha L^\alpha(|X|+A)\right)<\infty \text{ if and only if } E\left(f(|X|+A)\right)<\infty.\]
For a non negative random variable $Y$, $EY<\infty$ if and only if $\sum_{n=1}^\infty P(Y>n)<\infty$.
Applying this, we have that 
$E\left(f(|X|+A)\right)<\infty$ if and only if
\begin{equation}\label{r01}
\sum_{n=1}^\infty P\left(f(|X|+A)>n\right)<\infty.
\end{equation}
By using \lemref{lemma00sv} with $\beta=1,$ we have $f(g(x))\sim g(f(x)) \sim x$ as $x\to\infty$.
Combining this with the assumption that 
$f(x)$ and $g(x)$
are increasing on $[A,\infty)$, we see that \eqref{r01} is equivalent to
\begin{equation}\label{r02}
\sum_{n\ge A^\alpha} P\left(|X|>n^{1/\alpha}\tilde{L}(n^{1/\alpha})\right)<\infty.
\end{equation}
The proof of the proposition is completed.
\end{proof}

\section{Complete Convergence for Weighted Sums of Negatively Associated and Identically Distributed Random Variables}\label{NA}

In the following theorem, we establish
complete convergence for weighted sums of negatively associated and identically distributed random
variables.
\thmref{Theorem3.1}  is new even when
the random variables are i.i.d. 
A special case of this result comes close
to a solution of an open question of Chen and Sung \cite{ChenSung14}.
In subsequent derivations, the symbol $C$
denotes a generic positive constant whose value may be different for
each appearance.

\begin{thm}\label{Theorem3.1}
Let $1\le \alpha<2$,
$\{X,X_n, \, n \geq 1\}$ be a sequence of negatively associated and identically distributed random variables and 
 $L(\cdot)$ a slowly varying function defined on $[A,\infty)$ for some $A>0$.
When $\alpha=1$, we assume further that $L(x)\ge 1$ and is increasing on $[A,\infty)$.
Let $b_n=n^{1/\alpha}{\tilde{L}}(n^{1/\alpha})$, $n\ge A^\alpha $. Then the 
following four statements are equivalent.
\begin{description}

\item(i) The random variable $X$ satisfies
\begin{equation}\label{add007}
E(X)=0,\ E\left(|X|^\alpha L^\alpha(|X|+A)\right)<\infty.
\end{equation}
\item(ii) For every array of constants $\{a_{ni},n\ge 1,1\le i\le n\}$
satisfying 
\begin{equation}\label{sv001}
\sum_{i=1}^{n}a_{ni}^{2}\le Cn,\ n\ge 1,
\end{equation} 
we have
\begin{equation}\label{sv003}
\sum_{n\ge A^\alpha}
n^{-1}P\left(\max_{1\le k\le n}\left|
\sum_{i=1}^{k}a_{ni}X_i\right|>\varepsilon  b_n\right)<\infty \text{ for all }\varepsilon >0.
\end{equation}
\item(iii) 
\begin{equation}\label{sv003-01}
\sum_{n\ge A^\alpha}
n^{-1}P\left(\max_{1\le k\le n}\left|
\sum_{i=1}^{k}X_i\right|>\varepsilon  b_n\right)<\infty  \text{ for all }\varepsilon >0.
\end{equation}

\item(iv) The strong law of large numbers
\begin{equation}\label{add014}
\begin{split}
\lim_{n\to\infty}\dfrac{
\max_{1\le k\le n}\left|\sum_{i=1}^{k}X_i\right|}{b_n}=0\ \text{ a.s.}
\end{split}
\end{equation}
holds.
\end{description}
\end{thm}

\begin{proof}
For simplicity, we assume that $A^\alpha$ is an integer number since we can take $[A^\alpha]+1$ otherwise. 
By Lemmas \ref{lem:GS} and \ref{lemma01sv}, without loss of generality, we can assume that 
$x^{1/\alpha}\tilde{L}(x^{1/\alpha})$ and $x^{\alpha-1}L^\alpha(x)$ are increasing on $[A,\infty)$ 
and that $x^{\alpha-2}L^\alpha(x)$ is decreasing on $[A,\infty)$. We may also assume that $a_{ni}\ge 0$
since we can using the identity $a_{ni}=a_{ni}^+-a_{ni}^-$ in the general case.

Firstly, we prove the implication ((i) $\Rightarrow$ (ii)). Assume that \eqref{add007} holds and 
$\{a_{ni},n\ge 1,i\ge 1\}$ are constants satisfying \eqref{sv001}, we will prove that \eqref{sv003} holds. For $n\ge A^\alpha$, set
$$X_{ni}=-b_n I(X_i<-b_n)+X_iI(|X_i|\le b_n)+b_n I(X_i>b_n),\ 1\le i\le n,$$
and
$$S_{nk}=\sum_{i=1}^k \Big(a_{ni}X_{ni}-
E(a_{ni}X_{ni})\Big),\ 1\le k\le n.$$
Let $\varepsilon>0$ be arbitrary. For $n\ge A^\alpha$,
\begin{equation}\label{sv005}
\begin{split}
&P\left(\max_{1\le k\le n}\left|\sum_{i=1}^k
a_{ni} X_{i}\right|>\varepsilon  b_n\right)\le P\left(\max_{1\le k\le n}|X_k|>b_n\right)
+P\left(\max_{1\le k\le n}\left|\sum_{i=1}^k a_{ni}
X_{ni}\right|>\varepsilon  b_n\right)\\
&\le P\Big(\max_{1\le k\le n}|X_k|>b_n\Big)+P\Big(\max_{1\le k\le n}|S_{nk}|>\varepsilon
b_n- 
\sum_{i=1}^n \left|E(a_{ni}X_{ni})\right|\Big ).
\end{split}
\end{equation}
By the second half of \eqref{add007} and Proposition \ref{lemma_bound01}, we have
\begin{equation}\label{sv007}
\begin{split}
\sum_{n\ge A^\alpha} n^{-1}P\Big(\max_{1\le k\le n}
|X_k|>b_n\Big)&\le \sum_{n\ge A^\alpha} n^{-1}\sum_{k=1}^{n}P
\Big(|X_k|>b_n\Big)\\
&= \sum_{n\ge A^\alpha} P(|X|>b_n)<\infty.
\end{split}
\end{equation}
For $n\ge 1$, by the Cauchy-Schwarz inequality and \eqref{sv001}, 
\begin{equation}\label{sv009}
\begin{split}
\Big(\sum_{i=1}^{n}|a_{ni}|\Big)^2 &\le
n\Big(\sum_{i=1}^{n}a_{ni}^2\Big) \le Cn^2.
\end{split}
\end{equation}
For $n\ge A^\alpha$, the first half of \eqref{add007} and \eqref{sv009} imply that
\begin{equation}\label{sv37}
\begin{split}
\dfrac{\sum_{i=1}^{n}|E(a_{ni}X_{ni})|}{b_n}& \le \dfrac{\sum_{i=1}^{n}|a_{ni}|\left(\left|EX_{i}I(|X_i|\le b_n)\right|+b_nP(|X_i|>b_n)\right)}{b_n}\\
& \le \dfrac{Cn\left(\big|E(XI(|X|\le b_n))\big|+b_nP(|X|>b_n)\right)}{b_n}\\
& =\dfrac{Cn\left(\big|E(XI(|X|> b_n))\big|+b_nP(|X|>b_n)\right)}{b_n}\\
& \le \dfrac{Cn E|X|I(|X|> b_n)}{b_n}.
\end{split}
\end{equation}
For $n$ large enough
and for $\omega \in (|X|>b_n)$, we have
\begin{equation}\label{sv39}
\begin{split}
\dfrac{n}{b_n}&= \dfrac{n^{(\alpha-1)/\alpha}\tilde{L}^{\alpha-1}(n^{1/\alpha})}{\tilde{L}^{\alpha}(n^{1/\alpha})}\\
&= \dfrac{\left(n^{1/\alpha}\tilde{L}(n^{1/\alpha})\right)^{\alpha-1}L^\alpha \left(n^{1/\alpha}\tilde{L}(n^{1/\alpha})\right)}{\tilde{L}^{\alpha}(n^{1/\alpha})L^\alpha\left(n^{1/\alpha}\tilde{L}(n^{1/\alpha})\right)}\\
&\le C b_{n}^{\alpha-1}L^\alpha(b_n) \le  C|X(\omega)|^{\alpha-1} L^\alpha(|X(\omega)|),
\end{split}
\end{equation}
where we have applied \eqref{BGT1513} in the first inequality and the monotonicity of $x^{\alpha-1}L^\alpha(x)$ in the second inequality.
Combining \eqref{sv37}, \eqref{sv39}, the second half of \eqref{add007} and using Lemma \ref{lemma01sv} (ii), we have 
\begin{equation}\label{sv41}
\begin{split}
\dfrac{\sum_{i=1}^{n}|E(a_{ni}X_{ni})|}{b_n}&
\le CE\left(|X|^\alpha L^\alpha(|X|)I\left(|X|> b_n\right)\right)\\
&\le CE\left(|X|^\alpha L^\alpha(|X|+A)I\left(|X|> b_n\right)\right)\\
&\to 0 \text{ as } n\to \infty.
\end{split}
\end{equation}
From
\eqref{sv005}, \eqref{sv007} and \eqref{sv41}, to obtain \eqref{sv003}, it
remains to show that
\begin{equation}\label{sv013}
\sum_{n\ge A^\alpha}
n^{-1}P\Big(\max_{1\le j\le n} |S_{nj}|>b_n\varepsilon /2 \Big
)<\infty.
\end{equation}
Set $b_{A^\alpha-1}=0$.
For $B$ large enough, we have
\begin{equation}\label{sv012}
\begin{split}
&\sum_{n\ge A^\alpha} \dfrac{1}{n}P\Big(\max_{1\le k\le n}
|S_{nk}|>b_n\varepsilon /2 \Big )
  \le \sum_{n\ge A^\alpha}\dfrac{4 }{\varepsilon ^2 nb_{n}^2}
E\Big(\max_{1\le j\le n}|S_{nj}|\Big)^2\\
 & \le \sum_{n\ge A^\alpha}\dfrac{4 }{\varepsilon ^2 nb_{n}^2}
\sum_{i=1}^nE\Big(a_{ni}X_{ni}-E(a_{ni}X_{ni})\Big)^2\\
 & \le \sum_{n\ge A^\alpha}\dfrac{4 \left(\sum_{i=1}^na_{ni}^2\right)\left(EX^2I(|X|\le b_n)+b_{n}^2 P(|X|>b_n)\right)}{\varepsilon ^2 n b_{n}^{2}}
\\
 & \le C\sum_{n\ge A^\alpha}\left(\dfrac{E(X^2I(|X|\le b_n))}{b_{n}^{2}}+P(|X|>b_n)\right)\\
& \le C+ C\sum_{n\ge A^\alpha}\dfrac{1}{n^{2/\alpha}\tilde{L}^2(n^{1/\alpha})}\sum_{A^\alpha\le i\le n} E\left( X^2I(b_{i-1}<|X|\le b_{i})\right)\\
& = C+C\sum_{i\ge B}\left(\sum_{n\ge i}\dfrac{1}{n^{2/\alpha}\tilde{L}^2(n^{1/\alpha})}\right) E\left( X^2I(b_{i-1}<|X|\le b_{i})\right)\\
& \le C+C\sum_{i\ge B} i^{(\alpha-2)/\alpha}\tilde{L}^{-2}(i^{1/\alpha}) E\left( X^2I(b_{i-1}<|X|\le b_{i})\right),
\end{split}
\end{equation}
where we have used Chebyshev's inequality in the first inequality, the Kolmogorov maximal inequality (see Shao \cite[Theorem 2]{Shao00})
in the second inequality, \eqref{sv001} in the fourth inequality, \propref{lemma_bound01} and the second half of \eqref{add007} in the fifth inequality
and Lemma \ref{sv51} in the last inequality.
For $\omega \in (b_{i-1}<|X|\le b_{i})$, we have
\begin{equation}\label{sv014}
\begin{split}
i^{(\alpha-2)/\alpha}\tilde{L}^{-2}(i^{1/\alpha})&=\dfrac{i^{(\alpha-2)/\alpha}\tilde{L}^{\alpha-2}(i^{1/\alpha})L^\alpha\left(i^{1/\alpha}\tilde{L}(i^{1/\alpha})\right)}
{\tilde{L}^{\alpha}(i^{1/\alpha})L^\alpha\left(i^{1/\alpha}\tilde{L}(i^{1/\alpha})\right)}\\
&\le C\left(i^{1/\alpha}\tilde{L}(i^{1/\alpha})\right)^{\alpha-2}L^\alpha\left(i^{1/\alpha}\tilde{L}(i^{1/\alpha})\right)\\
&= Cb_{i}^{\alpha-2}L^\alpha\left(b_i\right)\le C|X(\omega)|^{\alpha-2} L^\alpha\left(|X(\omega)|\right),
\end{split}
\end{equation}
where we have applied \eqref{BGT1513} in the first inequality and the monotonicity of $x^{2-\alpha}L^\alpha(x)$ in the second inequality.
Combining \eqref{sv012}, \eqref{sv014}, the second half of \eqref{add007} and using Lemma \ref{lemma01sv} (ii), we have
\begin{equation}\label{sv016}
\begin{split}
\sum_{n\ge A^\alpha} \dfrac{1}{n}P\Big(\max_{1\le k\le n}
|S_{nk}|>b_n\varepsilon /2 \Big )
& \le C+ C E\left(|X|^\alpha L^\alpha(|X|+A)\right)<\infty,
\end{split}
\end{equation}
thereby proving \eqref{sv013}.

The implication [(ii) $\Rightarrow$ (iii)] is immediate by letting $a_{ni}\equiv 1$. Now, we assume that (iii) holds. Since 
\[b_n=n^{1/\alpha}\tilde{L}(n^{1/\alpha}) \uparrow \infty \text{ and }\dfrac{b_{2n}}{b_{n}}=\dfrac{2^{1/\alpha}\tilde{L}((2n)^{1/\alpha})}{\tilde{L}(n^{1/\alpha})}\le C,\]
it follows from the proof of \cite[Lemma 2.4]{Sung14} that (see (2.1) in \cite{Sung14})
\begin{equation}\label{S01}
\lim_{k\to\infty}\dfrac{\max_{1\le i\le 2^{k+1}}|\sum_{j=1}^i X_j|}{b_{2^k}}=0 \text{ a.s.}
\end{equation}
For $2^k\le n< 2^{k+1}$,
\begin{equation}\label{S02}
\dfrac{\max_{1\le i\le n}|\sum_{j=1}^i X_j|}{b_{n}}\le \dfrac{\max_{1\le i\le 2^{k+1}}|\sum_{j=1}^i X_j|}{b_{2^k}}.
\end{equation}
Combining \eqref{S01} and \eqref{S02}, we obtain \eqref{add014}.

Finally, we prove the implication [(iv)$\Rightarrow$(i)]. It follows from \eqref{add014}
that
\begin{equation}\label{add050}
\lim_{n\to\infty}\dfrac{\max_{1\le k\le n}\left| X_k\right|}{b_n}=0\ \text{a.s.}
\end{equation}
Since $\{X,X_n,n\ge 1\}$ is a sequence of negatively associated random variables, 
$\{(X_n>b_n),n\ge 1\}$ are pairwise negatively correlated events, and so are $\{(X_n<-b_n),n\ge 1\}$.
By the generalized Borel-Cantelli lemma (see, e.g., \cite{Petrov}), it follows from 
\eqref{add050} that
\begin{equation}\label{add052}
\sum_{n\ge A^\alpha}P(|X|>b_n)=\sum_{n\ge A^\alpha}P(|X_n|>b_n)<\infty
\end{equation}
which, by \propref{lemma_bound01}, is equivalent to
\begin{equation}\label{add016}
E\left(|X|^\alpha L^\alpha(|X|+A)\right)<\infty.
\end{equation}
From \eqref{add016}, we have $E|X|<\infty$. Since
$|X-EX| \le |X|+E|X|$ and $L(\cdot)$ 
is differentiable slowly varying, 
 \eqref{add016} further implies\begin{equation}\label{add016_b}
E\left(|X-EX|^\alpha L^\alpha(|X-EX|+A)\right)<\infty.
\end{equation}
From \eqref{add016_b} and the proof of
 ((i)$\Rightarrow$(iv)), we have
\begin{equation}\label{add020}
\begin{split}
\lim_{n\to\infty}\left(\dfrac{
\sum_{i=1}^{n} X_i}{b_n}-\dfrac{n^{(\alpha-1)/\alpha}EX}{\tilde{L}(n^{1/\alpha})}\right)&=\lim_{n\to\infty}\dfrac{
\sum_{i=1}^{n}(X_i-EX_i)}{b_n}\\
&=0\ \text{a.s.}
\end{split}
\end{equation}
For the case where $1<\alpha<2$, we have from Remark \ref{R11} that
$n^{(\alpha-1)/\alpha}/\tilde{L}(n^{1/\alpha})\rightarrow\infty$ as $n\to\infty$. For the case where $\alpha=1$, we have
from \eqref{BGT1513} that $n^{(\alpha-1)/\alpha}/\tilde{L}(n^{1/\alpha})=1/\tilde{L}(n)
\sim L(n\tilde{L}(n))\ge 1$. 
It thus follows from
\eqref{add014} and \eqref{add020} that $E(X)=0$, i.e., the first half of \eqref{add007} holds.
The proof is completed.
\end{proof}

By letting $L(x)\equiv 1$, \thmref{Theorem3.1} generalizes a seminal result of
Baum and Katz \cite{BaKa} on complete convergence
for sums of independent random variables to weighted sums of negatively associated random variables.
Recently, Miao et al. \cite{MMX}
proved the following proposition.

\begin{prop}[\cite{MMX}, Theorem 2.1]\label{prop:MMX}
Let $0<\alpha<2$ and let $\{X,X_n,n\ge 1\}$ be a strictly stationary negatively associated sequence
with $E|X|^\alpha\log^{-\beta}(|X|+2)<\infty$ for some $\beta\ge 0$. In the case where $1< \alpha<2$, 
assume further that
$EX=0$. Then for any $\delta\ge (1+\alpha^2-\alpha)/\alpha$, we have
\begin{equation}\label{MMX01}
\lim_{n\to\infty}\dfrac{
\sum_{i=1}^{n}X_i}{n^{1/\alpha}(\log n)^{\beta(1-\alpha+\delta)}}=0\ \text{ a.s.}
\end{equation}
\end{prop}

We observe that one only needs to verify \eqref{MMX01} for the case where $\delta=(1+\alpha^2-\alpha)/\alpha$. 
In this case, \eqref{MMX01} becomes
\begin{equation}\label{MMX02}
\lim_{n\to\infty}\dfrac{
\sum_{i=1}^{n}X_i}{n^{1/\alpha}(\log n)^{\beta/\alpha}}=0\ \text{ a.s.}
\end{equation}
For the case where (i) $1<\alpha<2,\beta\ge 0$ or (ii) $\alpha=1,\beta=0$, by letting $L(x)\equiv \log^{-\beta/\alpha}(x+2)$, 
we see that \eqref{add014} reduces to \eqref{MMX02}. Therefore,  Proposition \ref{prop:MMX}
is a special case of \thmref{Theorem3.1}. For the case where $\alpha=1$ and $\beta>0$, we will show in the next section (Sect. \ref{Pair}) that
Proposition \ref{prop:MMX} holds under a weaker condition that $\{X,X_n,n\ge 1\}$ are pairwise negatively dependent.

Now, we consider another special case where $1<\alpha<2$, $\gamma>0$ and
$L(x)=\log^{-1/\gamma}(x),\ x\ge 2$.
Then $$b_n=n^{1/\alpha}L^{-1}(n^{1/\alpha})=\left(\dfrac{1}{\alpha}\right)^{1/\gamma}n^{1/\alpha}\log^{1/\gamma}(n),\ n\ge 2,$$ 
and we have the following corollary.
This result comes close to a solution to the open question raised by
Chen and Sung \cite{ChenSung14} which we have mentioned in  Introduction.

\begin{cor}\label{cor01}
Let $1< \alpha<2$,  $\gamma>0$ and 
$\{X,X_n, \, n \geq 1\}$ be a sequence of negatively associated and identically distributed random variables.
Then the following statements are equivalent.
\begin{description}

\item(i) The random variable $X$ satisfies
\begin{equation*}
E(X)=0\ \text{ and }\ E\left(|X|^\alpha/\log^{\alpha/\gamma}(|X|+2)\right)<\infty.
\end{equation*}
\item(ii) For every array of constants $\{a_{ni},n\ge 1,1\le i\le n\}$
satisfying \eqref{sv001}, we have \eqref{ChenSung03}.
\item(iii) The strong law of large numbers
\begin{equation*}
\begin{split}
\lim_{n\to\infty}\dfrac{
\max_{1\le k\le n}\left|\sum_{i=1}^{k}X_i\right|}{n^{1/\alpha}\log^{1/\gamma}n}=0\ \text{ a.s.}
\end{split}
\end{equation*}
holds.
\end{description}
\end{cor}

\begin{rem} {\rm
When $\alpha\le 2$, by H\"{o}lder's inequality, our condition \eqref{sv001} implies \eqref{ChenSung05}.
From \cororef{cor01}, we see that by slightly extending \eqref{ChenSung05}, 
we obtain optimal moment condition for \eqref{ChenSung03}.
}
\end{rem}

In the following example, we show that the moment condition provided
by Chen and Sung \cite[Corollary 2.2]{ChenSung14} is violated, but \cororef{cor01} can still be applied.

\begin{exm}\label{exm01}{\rm Let $1< \alpha< 2$, $\gamma>0$ and 
$\{X,X_n,n\ge 1\}$ be a sequence of negatively associated and identically distributed random variables
with the common density function
\[ f(x)= \dfrac{b }{|x|^{\alpha + 1}\log^{1-\alpha/\gamma}(|x|+2)\log^2(\log(|x|+2))}I(|x| > 1),\]
where $b$ is the normalization constant. 
Then 
\[EX=0,\ E\left(|X|^\alpha/\log^{\alpha/\gamma}(|X|+2)\right)<\infty.\]
Therefore, by applying \cororef{cor01} with $a_{ni}\equiv 1$, we obtain 
\[\sum_{n=1}^\infty 
 n^{-1} P\left(\max_{1\le k\le n}\left|
\sum_{i=1}^{k}X_i\right|>\varepsilon  n^{1/\alpha}\log^{1/\gamma}(n)\right)<\infty
\ \text{ for all } \ \varepsilon >0,\]
and
\begin{equation*}
\begin{split}
\lim_{n\to\infty}\dfrac{
\max_{1\le k\le n}\left|\sum_{i=1}^{k}X_i\right|}{n^{1/\alpha}\log^{1/\gamma}n}=0\ \text{ a.s.}
\end{split}
\end{equation*}
In this example, we cannot apply Corollary 2.2 in Chen and Sung \cite{ChenSung14} since
\begin{equation}\label{eq:add10}
E\left(|X|^\alpha\log^{1-\alpha/\gamma}(|X|+2)\right)=\infty.
\end{equation}
}	
	\end{exm}
	
\section{Strong Law of Large Numbers for Sequences of Pairwise Negatively Dependent and Identically Distributed Random Variables}\label{Pair}

For a sequence of i.i.d. random variables $\{X,X_n,n\ge 1\}$,
the classical Hartman–Wintner law of the iterated logarithm states that $E(X)=0$ and
$E(X^2)<\infty$ are necessary and sufficient conditions for
the law of the iterated logarithm to hold.

By letting $L(x)\equiv 1$ in \thmref{Theorem3.1}, we see that the Marcinkiewicz--Zygmund strong law of large numbers holds
for sequences of negatively associated and identically distributed random variables
under optimal condition $E|X|^\alpha<\infty$. 
However, in \thmref{Theorem3.1}, for the case where $\alpha=1$, we require $L(x)\ge 1$ for $x\ge A$.
The reason behind this is because we need $E|X|<\infty$ in the proof. 
The aim of this section is
to establish the strong law of large numbers for the case
where $E|X|=\infty$. It turns out that a similar strong law of large numbers
still holds even for pairwise negatively dependent random variables. 
On the law of the iterated logarithm, this line of research
was initiated by Feller \cite{Feller68a} and completely developed by
Kuelbs and Zinn \cite{KZinn83}, Einmahl \cite{Einmahl93}, 
Einmahl and Li \cite{EL05,EL08} where
the authors proved general laws of the iterated logarithm for
sequences of i.i.d. random variables
with $E(X^2)=\infty$. The
normalizing sequences in
laws of the iterated logarithm in 
Einmahl and Li \cite{EL05,EL08} are also of
the form $\sqrt{nL(n)}$, where $L(n)$ is a
slowly varying increasing function.

It is worth noting that
for pairwise i.i.d. random variables, 
the Marcinkiewic-Zygmund strong law of large numbers holds under
optimal moment condition $E|X|^\alpha<\infty$, $1\le \alpha<2$ (see Etemadi \cite{Etemadi81} for the case where $\alpha=1$ and
Rio \cite{Rio95b} for the case where $1<\alpha<2$). 
On the case where the random variables are pairwise independent, but not identically distributed,  Cs\"{o}rg\H{o}
et al. \cite{CTT} proved that the Kolmogorov condition alone does not ensure the strong law of large numbers.
Bose and Chandra \cite{BoseChandra}, and Chandra and Goswami \cite{ChandraGoswami03}
generalized the Marcinkiewicz--Zygmund-type law of large numbers for pairwise independent case under 
the so-called Ces\`{a}ro uniform integrability condition.

For pairwise negatively dependent random variables, Shen et al. \cite{SZV}
established a strong law of large numbers for pairwise negatively dependent and identically distributed
 random variables under a very general condition.
Precisely, Shen et al. \cite[Theorems 3 and 5]{SZV} proved that if $\{b_n,n\ge 1\}$ is a sequence of 
positive constants with $b_n/n\uparrow \infty$
and if $\{X,X_n,n\ge 1\}$ is a sequence of pairwise negatively dependent and identically distributed random variables, then 
$\sum_{n=1}^\infty P(|X|>b_n)<\infty$ if and only if $\sum_{n=1}^\infty n^{-1} P\left(\max_{1\le k\le n}|\sum_{i=1}^kX_i|>b_n\varepsilon\right)<\infty$
for all $\varepsilon>0$. 
By combining this result of Shen et al. \cite{SZV} with \propref{lemma_bound01}, we have the following theorem.
 
\begin{thm}\label{Theorem3.2}
Let $\{X,X_n, \, n \geq 1\}$ be a sequence of pairwise negatively dependent and
identically distributed random variables, and 
let $L(\cdot)$ be a slowly varying function defined on $[A,\infty)$ for some $A>0$ with
$\tilde{L}(x)\uparrow \infty$ as $x\to\infty$. 
Then the 
following statements are equivalent.
\begin{description}

\item(i) The random variable $X$ satisfies
\begin{equation}\label{add007_cor_pair}
E\left(|X| L(|X|+A)\right)<\infty.
\end{equation}
\item(ii) 
\begin{equation}\label{sv003_cor_pair}
\sum_{n\ge A}
n^{-1}P\left(\max_{1\le k\le n}\left|
\sum_{i=1}^{k}X_i\right|>\varepsilon  n\tilde{L}(n)\right)<\infty \text{ for all }\varepsilon >0.
\end{equation}
\item(iii) The following strong law of large numbers holds:
\begin{equation}\label{add014_pair}
\begin{split}
\lim_{n\to\infty}\dfrac{
\sum_{i=1}^{n}|X_i|}{n\tilde{L}(n)}=0\ \text{ a.s.}
\end{split}
\end{equation}
\end{description}
\end{thm}

Martikainen \cite{Martikainen} proved that if $\{X,X_n,n\ge 1\}$ is a sequence of pairwise i.i.d. mean $0$ random variables, then 
$E|X|\log^{\gamma}(|X|+2) <\infty$ for some $\gamma>0$ 
if and only if $\lim_{n\to\infty}\dfrac{
\sum_{i=1}^{n}X_i}{n \log^{-\gamma}(n)}=0\ \text{ a.s.}$
In \thmref{Theorem3.2}, by letting $L(x)\equiv \log^{-\gamma}(x)$ for some $\gamma>0$, then 
for sequences of pairwise negatively dependent and identically distributed random variables, we have
$E|X|\log^{-\gamma}(|X|+2) <\infty$ 
if and only if $\lim_{n\to\infty}\dfrac{
\sum_{i=1}^{n}X_i}{n \log^{\gamma}(n)}=0\ \text{ a.s.}$
Therefore, a very special case of \thmref{Theorem3.2} can be considered as a counterpart of the main result in 
Martikainen \cite{Martikainen}. This special case also extends Proposition \ref{prop:MMX} (for the case where $\alpha=1,\beta>0$) to pairwise negatively dependent
random variables.

Finally, we present the following example to illustrate \thmref{Theorem3.2}.
This example concerns a random variable appearing in the St. Petersburg game.

\begin{exm} {\rm The St. Petersburg game
which is defined as follows: Tossing a
fair coin repeatedly until the head appears. If this happens at trial number $n,$
you receive $2^n$ Euro.  
The random variable $X$ behind the game has probability mass function:
\begin{equation}\label{Peter00}
P(X=2^n)=\dfrac{1}{2^n},n\ge 1.
\end{equation}
Since $E(X)=\infty$, a fair price for you
to participate in the game would be impossible. 
To set the fee as a function of the number of games, 
Feller \cite[Chapter X]{Feller68} (see also in Gut \cite{Gut04}) proved that
\begin{equation}\label{Peter01}
\lim_{n\to\infty}\dfrac{\sum_{i=1}^n X_i}{n\log n}=1 \text{ in probability},
\end{equation}
where $\{X_n,n\ge 1\}$ are independent random variables which have
the same distribution as $X$. 

By Theorem 2 of Chow and Robbins \cite{ChowRobbins}, it is impossible to have 
almost sure convergence in \eqref{Peter01}.
The natural question that comes to mind is what
would be an ``optimal'' (or ``smallest'') choice of
$\{b_n,n\ge 1\}$ in order for
\[\lim_{n\to\infty}\dfrac{\sum_{i=1}^n X_i}{b_n}=0 \text{ a.s.}\]
to hold? 
It turns out that we can
have such a strong law of large numbers even by
requiring only the random variables $\{X_n,n\ge 1\}$ are pairwise negatively dependent
and have the same distribution as $X$.
To see this, let
\[L(x)=\left((\log|x|)(\log(\log (4+|x|)))^{1+\gamma}\right)^{-1},\]
where $\gamma$ is positive, arbitrary small, but fixed, 
then $E\left(|X| L(|X|)\right)<\infty$.
By \thmref{Theorem3.2}, the Borel-Cantelli lemma and some
easy computations, we can show that
\begin{equation}\label{Peter05}
\lim_{n\to\infty}\dfrac{\sum_{i=1}^n X_i}{n(\log n)(\log(\log (4+n)))^{1+\gamma}}=0 \text{ a.s.,}
\end{equation}
and
\begin{equation}\label{Peter07}
\limsup_{n\to\infty}\dfrac{\sum_{i=1}^n X_i}{n(\log n)(\log(\log (4+n)))\log(\log(\log (4+n)))}=\infty \text{ a.s.}
\end{equation}
}
\end{exm}

\begin{rem} {\rm For the i.i.d. case, Cs\"{o}rg\H{o} and Simons \cite{CsorgoSimons} obtained \eqref{Peter05} and \eqref{Peter07} 
by applying their strong law of large numbers for trimmed sums.
}
\end{rem}

	\textbf{Acknowledgments.}
The authors are grateful to the referee for constructive, perceptive
and substantial comments and suggestions which enabled us to greatly improve the paper. 
In particular, the referee's suggestions on the asymptotic inverse of the regularly varying function $x^\alpha L(x)$ enabled us
to obtain Theorems \ref{Theorem3.1} and \ref{Theorem3.2} which are considerably more general than
those of the initial version of the paper. 


{\small
	\begin{thebibliography}{99}


		
		\bibitem{BaKa}
		Baum, L.E., Katz,  M.: Convergence rates in
		the law of large numbers.
		Trans. Amer. Math. Soc. {\bf 120},  108--123 (1965)
		
\bibitem{BoseChandra}
Bose, A., Chandra, T. K.:
Ces\`{a}ro uniform integrability and $L_p$-convergence.
Sankhya Ser. A \textbf{55}, 12--28 (1993)
		
		\bibitem{BGT}
		Bingham, N. H.,  Goldie, C. M.,   Teugels, J. L.: Regular variation (Encyclopedia of Mathematics and its Applications, 27). Cambridge University Press, Cambridge (1989)
		
		 \bibitem{BSS}{}
Block, H. W., Savits, T. H., Shaked, M.:
Some concepts of negative dependence, Ann. Probab., \textbf{10}, no. 3, 765--772 (1982)

		
		\bibitem{Bojanic71}
		Bojani\'{c}, R., Seneta, E.: Slowly varying functions and asymptotic relations. {J. Math. Anal. Appl.} \textbf{34}, 302--315 (1971)
		
		
 \bibitem{BDD}{}%
 Burton, R. M., Dabrowski A. R., Dehling,  H.:
An invariance principle  for weakly associated random vectors,
Stochastic Process. Appl. \textbf{23}, 301--306 (1986)


\bibitem{ChandraGoswami03}		
Chandra, T. K., Goswami, A.: Ces\`{a}ro $\alpha$-integrability and laws of large numbers. I. J. Theoret. Probab. \textbf{16}, no. 3, 655--669 (2003)

		
		
		\bibitem{ChenSung14}
		Chen, P.,  Sung,  S. H.: On the strong convergence for weighted sums of negatively associated random variables. 
		{Statist. Probab. Lett.} \textbf{92}, 45--52 (2014)
		

\bibitem{ChowRobbins}
Chow, Y. S., Robbins, Herbert: 
On sums of independent random variables with infinite moments and ``fair'' games. Proc. Nat. Acad. Sci. U.S.A. 
\textbf{47}, 330--335 (1961)

\bibitem{CsorgoSimons}{}
Cs\"{o}rg\H{o}, S., Simons, G. A strong law of large numbers for trimmed sums, with applications to generalized St. Petersburg games. 
Statist. Probab. Lett. \textbf{26}, no. 1, 65--73.  (1996)

\bibitem{CTT}{}
Cs\"{o}rg\H{o}, S., Tandori, K., Totik, V.: On the strong law of
large numbers for pairwise independent random variables, { Acta
Math. Hungar.,} \textbf{42}, no. 3--4, 319--330 (1983)

\bibitem{DedeckerMerlevede}{}
		Dedecker, J.,   Merlev\`{e}de, F.: 
		Convergence rates in the law of large numbers for Banach-valued dependent variables. {Theory Probab. Appl.} \textbf{52}, no. 3, 416--438 (2008)	
		
\bibitem{EG81}{}
 Ebrahimi, N., Ghosh, M.: Multivariate negative dependence. 
Comm. Statist. A - Theory Methods. \textbf{10}, no. 4, 307--337 (1981)

\bibitem{Einmahl93}
Einmahl, U: Toward a general law of the iterated logarithm in Banach space. Ann. Probab. \textbf{21}, no. 4, 2012--2045 (1993)

\bibitem{EL05} 
Einmahl, U., Li, D-L.: Some results on two-sided LIL behavior. Ann. Probab. \textbf{33}, no. 4, 1601--1624 (2005)


\bibitem{EL08} 
Einmahl, U., Li, D-L.: Characterization of LIL behavior in Banach space. Trans. Amer. Math. Soc. \textbf{360}, no. 12, 6677--6693 (2008)

	
\bibitem{Etemadi81}
Etemadi, N.: An elementary proof of the strong law of large numbers. 
Z. Wahrsch. Verw. Gebiete \textbf{55}, no. 1, 119--122  (1981)


\bibitem{Feller68a}
Feller, W.: An extension of the law of the iterated logarithm to variables without
variance. J. Math. Mech. \textbf{18}, 343--355 (1968)
		
\bibitem{Feller68}
Feller, W.: \textit{ An Introduction to Probability Theory and Its Applications}, Vol 1, 3rd edn. John Wiley, New
York (1968)

		\bibitem{GS73}  Galambos, J.,   Seneta, E.:
		Regularly varying sequences.  {Proc. Amer. Math. Soc.} \textbf{41}, 110--116 (1973)
		
		
		
\bibitem{Gut04} 
		Gut, A.: An extension of the Kolmogorov-Feller weak law of large numbers with an application to the St. Petersburg game. 
		J. Theoret. Probab. \textbf{17}, no. 3, 769--779  (2004)


	\bibitem{GutStadmueller}
		Gut, A.,  Stadtm\"{u}ller,  U.:
		On the strong law of large numbers for delayed sums and random fields. 
		{Acta Math. Hungar.} \textbf{129}, no. 1-2, 182--203 (2010) 
		
		\bibitem{HRS}
		Heath, D.,  Resnick,  S.,  Samorodnitsky, G.: Heavy tails and long range dependence in ON/OFF processes and associated fluid models. {Math. Oper. Res.} \textbf{23}, 145--165 (1998)
		
	
		\bibitem{HechnerHeinkel}
		Hechner, F.,  Heinkel, B.: The Marcinkiewicz--Zygmund LLN in Banach spaces:
		a generalized martingale approach. {J. Theoret. Probab.} \textbf{23} (2010), no. 2, 509--522.
		
		\bibitem{HTV}
Hien, N. T. T., Thanh, L. V., Van, V. T. H., On the negative dependence in Hilbert spaces with applications. Appl. Math. 
\textbf{64}, no. 1, 45--59 (2019)

		\bibitem{JeMi}
		Jessen, A. H.,  Mikosch, T.: Regularly varying functions. {Publ. Inst. Math. (Beograd) (N.S.)} {\bf80}, 171--192 (2006)
		
		\bibitem{JiLi}
		Jing, B. Y., Liang,  H. Y.: Strong limit theorems for weighted sums
		of negatively associated random variables. {J. Theoret. Probab.}
		{\bf21}, 890--909  (2008)
		
		\bibitem{JoPr}
		Joag-Dev, K.,  Proschan, F.: Negative association of random
		variables, with applications. {Ann. Statist.} {\bf11}, 286--295 (1983)


		

\bibitem{KKH}
Ko, M. H., Kim T. S., Han, K. H.: A note on the almost
sure convergence for dependent random variables in a Hilbert space,
J. Theoret. Probab. \textbf{22}, 506--513 (2009)

		
\bibitem{KZinn83}
		Kuelbs, J., Zinn, J.: Some results of LIL behavior. Ann. Probab. \textbf{11}, no. 3, 506--557 (1983)


\bibitem{Lehmann66}
 Lehmann, E. L.: Some concepts of dependence, Ann. Math. Statist., \textbf{37}, 1137--1153 (1966)


		\bibitem{Martikainen}
		Martikainen, A. I.:  A remark on the strong law of large numbers for sums of pairwise independent random variables. 
		{J. Math. Sci.} \textbf{75}, no. 5, 1944--1946 (1995)
		
\bibitem{MN13}
Matsumoto, K., Nakata, T.: Limit theorems for a generalized Feller game. J. Appl. Probab. \textbf{50}, no. 1, 54--63  (2013)
		
		\bibitem{Matula92}
		Matula, P.:  A note on the almost sure convergence of sums of
		negatively dependent random variables.  {Statist. Probab. Lett.}
		\textbf{15}, 209--213 (1992)
		
		
		\bibitem{MMX}
		Miao, Y.,   Mu, J., Xu, J.: An analogue for Marcinkiewicz--Zygmund strong law of negatively associated random variables.
		{Rev. R. Acad. Cienc. Exactas F\'{i}s. Nat. Ser. A Mat. RACSAM.} \textbf{111}, no. 3, 697--705 (2017)
		
		\bibitem{MRRS}
		Mikosch, T.,  Resnick, S.,  Rootz\'{e}n, H.,   Stegeman, A.: Is network traffic approximated by stable L\'{e}vy motion or fractional Brownian motion?
		{Ann. Appl. Probab.} \textbf{12},
		23--68 (2002)
		

		\bibitem{Petrov}
		Petrov, V.V.: A note on the Borel-Cantelli lemma.  {Stat. Prob. Lett.}, \textbf{58},
		283–286 (2002)
		
		\bibitem{Rio95a}
		Rio, E.: A maximal inequality and dependent Marcinkiewicz--Zygmund strong laws. 
		{Ann. Probab.} \textbf{23}, no. 2, 918--937 (1995)
		
\bibitem{Rio95b}
Rio, E.: Vitesses de convergence dans la loi forte pour des suites dépendantes. (French) [Rates of convergence in the strong law for dependent sequences] C. R. Acad. Sci. Paris Sér. I Math. \textbf{320}, no. 4, 469--474  (1995)

\bibitem{Seneta73}
	Seneta, E.: An interpretation of some aspects of Karamata's theory of regular variation. { Publ. Inst. Math. (Beograd) (N.S.)},
	\textbf{15}, 111--119 (1973)
		
		\bibitem{Seneta76}
		Seneta, E.:  {Regularly varying functions}. Lecture Notes in Mathematics, Vol. 508. Springer-Verlag, Berlin-New York (1976)
		
		\bibitem{Shao00}
		Shao, Q. M.: A comparison on maximum inequalites between negatively associated
		and independent random variables.  {J. Theort. Probab.}, \textbf{13}, 343--356 (2000)

\bibitem{SZV}
Shen, A., Zhang, Y., Volodin, A.: On the strong convergence and complete convergence for pairwise NQD random variables.
 Abstr. Appl. Anal., Art. ID 949608, 7 pp (2014)
		
		\bibitem{Sung11}{}%
		Sung, S. H.: On the strong convergence for weighted sums of random variables. 
		{Statist. Papers.} \textbf{52}, no. 2, 447--454 (2011)
		


	\bibitem{Sung14}{}%
		Sung, S. H.: Marcinkiewicz--Zygmund type strong law of large numbers for pairwise i.i.d. random variables. 
		{J. Theoret. Probab.} \textbf{27}, no. 1, 96--106 (2014)

\bibitem{Szewczak}
Szewczak, Z.: On Marcinkiewicz--Zygmund laws. {J. Math. Anal. Appl.}, \textbf{375}, no. 2, 738--744 (2011)

\bibitem{Thanh13} 
Thanh, L. V.: On the almost sure convergence for dependent random vectors in Hilbert spaces,
Acta Math. Hungar. \textbf{139}, no. 3, 276--285 (2013)

	\end{thebibliography}
}
\end{document}